\documentclass[11pt]{amsart}
\usepackage{amsmath}
\usepackage{}
\usepackage{amsfonts}
\usepackage{amsfonts,latexsym,rawfonts,amsmath,amssymb,amsthm}
\usepackage[plainpages=false]{hyperref}

\usepackage{graphicx}
\makeatletter
\renewcommand\@biblabel[1]{}
\makeatother

\numberwithin{equation}{section}

\newcommand{\beq}{\begin{equation}}
\newcommand{\eeq}{\end{equation}}
\newcommand{\beqs}{\begin{eqnarray*}}
\newcommand{\eeqs}{\end{eqnarray*}}
\newcommand{\beqn}{\begin{eqnarray}}
\newcommand{\eeqn}{\end{eqnarray}}
\newcommand{\beqa}{\begin{array}}
\newcommand{\eeqa}{\end{array}}

\def\lra{\longrightarrow}

\def\bc{\begin{center}}
\def\ec{\end{center}}

\def\begeq{\begin{equation}}
\def\endeq{\end{equation}}
\def\and{\quad{\rm and}\quad}

\let\lra=\longrightarrow

\def\mapright\#1{\,\smash{\mathop{\lra}\limits^{\#1}}\,}

\newtheorem{prop}{Proposition}[section]
\newtheorem{theo}[prop]{Theorem}
\newtheorem{lem}[prop]{Lemma}
\newtheorem{claim}[prop]{Claim}
\newtheorem{cor}[prop]{Corollary}

\newtheorem{defi}[prop]{Definition}

\begin{document}

\date{}
\author {Yuxing $\text{Deng}^{*}$ }
\author { Xiaohua $\text{Zhu}^{**}$}

\thanks {*Partially supported by the NSFC Grants 11701030, **by the NSFC Grants 11331001 and 11771019}

\address{ Yuxing Deng\\School of Mathematics and Statistics, Beijing Institute of Technology,
Beijing, 100081, China\\
6120180026@bit.edu.cn}

\address{ Xiaohua Zhu\\School of Mathematical Sciences and BICMR, Peking University,
Beijing, 100871, China\\
xhzhu@math.pku.edu.cn}

\title{A note on compact $\kappa$-solutions of K\"{a}hler-Ricci flow}
\maketitle

\section*{\ }

\begin{abstract} In this paper, we give a complete classification of $\kappa$-solutions of K\"{a}haler-Ricci flow on compact complex manifolds. Namely, they must be  quotients of products of irreducible compact Hermitian symmetric manifolds.
\end{abstract}

\section{Introduction}

As an important problem in the analysis of the singularities of Ricci flow, the classification of ancient solutions to Ricci flow has attracted many attentions. 2-dimensional $\kappa$-solutions was classified by Perelman in \cite{Pe1}.  2-dimensional ancient Ricci flows with bounded curvature has also been classified even without the $\kappa$-noncollapsed condition by Daskalopoulos-Hamilton-Sesum (cf. \cite{DS},\cite{DHS}). Very recently, Brendle classfied 3-dimensional noncompact $\kappa$-solutions \cite{Br}. It remains an open problem for higher dimensional $\kappa$-solutions and 3-dimensional compact $\kappa$-solutions. It is also interesting to ask the same question for $\kappa$-solutions to K\"{a}hler-Ricci flow.  For eternal K\"{a}hler-Ricci flows with unformly bounded and nonnegative holomorphic bisectional curvature, the authors  have shown that they must be flat \cite{DZ2,DZ3}. In this paper, we give a complete classification of compact $\kappa$-solutions to K\"{a}haler-Ricci flow by using a non-existence result of steady K\"{a}hler-Ricci soliton in \cite{DZ3}. Namely, we prove

\begin{theo}\label{main theorem}
Any  $\kappa$-solution of K\"{a}hler-Ricci flow must be  a quotient of product of some irreducible compact Hermitian symmetric manifolds.
\end{theo}

By the definition of compact $\kappa$-solution of K\"{a}hler-Ricci flow $(M,g(t))$ (cf. Definition \ref{def1}), $(M,g(t))$ is a compact K\"{a}hler manifold with nonnegative holomorphic bisectional curvature. The classification of such manifolds is known as the generalized Frankel conjecture. The conjecture was proved by Mok \cite{M}. There are also many works related to the conjecture such as \cite{CaCh}, \cite{G} etc.. The generalized Frankel conjecture is stated as follows.

\begin{theo}\label{Frankel conjecture}
Let $(M,g)$ be an $n$-dimensional compact K\"{a}hler manifold of nonnegative holomorphic bisectional curvature and let $(\widetilde{M},\widetilde{g})$ be the its universal cover. Then, there exist nonnegative integers $k,N_1,\cdots,N_l$ and irreducible compact Hermitian symmetric spaces $M_1,\cdots,M_r$ of $rank\ge2$ such that such that $(\widetilde{M},\widetilde{g})$ is isometrically biholomorphic to
\begin{align}
(\mathbb{C}^k,g_0)\times(\mathbb{CP}^{n_1},g_1)\times\cdots\times(\mathbb{CP}^{n_l},g_l)\times(M_1,h_1)\times\cdots\times(M_r,h_r)
\end{align}
where $g_0$ is the Euclidean metric on $\mathbb{C}^k$, $h_1,\cdots,h_r$ are canonical metrics on $M_1,\cdots,M_r$ and $g_i(1\le i\le l)$ are K\"{a}hler metrics on $\mathbb{CP}^{n_i}$ with nonnegative holomorphic bisectional curvature.
\end{theo}

 With the help of the generalized Frankel conjecture,
we can classify compact backward type I $\kappa$-solutions similar to a result of Ni for $\kappa$-solutions with positive curvature operator  in \cite{Ni} (cf. Proposition \ref{prop-type I case}).  On the other hand, based on our former work \cite{DZ3},  we will  show that  any compact $\kappa$-solution  of K\"{a}hler-Ricci flow must be of Type I (cf. Lemma \ref{lem-flow is type I}).  Thus  Theorem \ref{main theorem} follows from Proposition \ref{prop-type I case} and Lemma \ref{lem-flow is type I}.

\section{Proof of Theorem \ref{main theorem}}
\begin{defi}\label{def1}
A complete K\"{a}hler-Ricci flow $(M, g(t))$ on $t\in (-\infty,0]$ is called ancient if the bisectional curvature of $g(t)$ is bounded
and nonnegative for any $t\in (-\infty, 0]$. A complete K\"{a}hler-Ricci flow $(M, g(t))$ is called a $\kappa$-solution to K\"{a}hler-Ricci flow if it is a $\kappa$-noncollapsed, non-flat ancient solution .
\end{defi}

We first recall a theorem on the character  of Type II limit solutions by Cao in \cite{Ca1}, which is

\begin{theo}\label{character-eternal-solution}
Any simply connected Type II limit solution of  K\"{a}hler-Ricci flow  on a simply connected complex manifold must be a gradient steady K\"{a}hler-Ricci soliton.
\end{theo}

The Type II limit solution in Theorem \ref{character-eternal-solution} is an eternal solution of K\"{a}hler-Ricci flow with uniformly bounded and nonnegative  bisectional curvature and positive Ricci curvature where the scalar curvature assumes its maximum in space-time. The assumption of  scalar curvature  is only used to make sure that there is a point $(x_0,t_0)$ such that
\begin{align}\label{eq:condition}
\frac{\partial R}{\partial t}(x_0,t_0)=0.
\end{align}
Moreover,  the ancient condition is sufficent  in Cao's argument. Hence, Cao's argument in \cite{Ca1} implies the following corollary.

\begin{cor}\label{cao-theorem-1}
   Let $(M,g(t))$ be an ancient solution of K\"{a}hler-Ricci flow with uniformly bounded nonnegative  bisectional curvature and positive Ricci curvature. If there is a point $(x_0,t_0)$ such that (\ref{eq:condition}) holds,  then it is a gradient steady K\"{a}hler-Ricci soliton.
\end{cor}

Now, we introduce the notation of backward type I Ricci flow.
\begin{defi}
An ancient solution $(M,g(t))$ defined on $M\times(-\infty,A)$ is called of backward type I if there exists a constant $C>0$ such that
\begin{align}
|{\rm Rm}|(x,t)\le \frac{C}{|t|}, ~\forall~t\le t_0,
\end{align}
where $t_0<0$ is a constant.
\end{defi}

\begin{lem}\label{lem-flow is type I}
Suppose that  $(M,g(t))$ is a $\kappa$-solution of K\"{a}haler-Ricci flow on compact complex manifold $M$. Then, $(M,g(t))$ is a backward type I $\kappa$-solution.
\end{lem}
\begin{proof}
We prove by contradiction. Let $M(t)=\sup_{x\in M}R(x,t)$. If the lemma is not true, then
\begin{align}\label{type I condition}
\limsup_{t\to-\infty}(-t)M(t)=\infty.
\end{align}

\begin{claim}
 Under the assumption of (\ref{type I condition}), there exist a time sequence $t_i\to-\infty$ such that
\begin{align}\label{quotient to zero}
\lim_{i\to\infty}\frac{\partial}{\partial t}R(p_i,t_i)\cdot R^{-2}(p_i,t_i)=0,
\end{align}
where $p_i$ assume the maximum of $R(x,t)$ at time slice $t=t_i$, i.e., $R(p_i,t_i)=M(t_i)$.
\end{claim}

We will use a trick in \cite{Br} to prove the claim. Note that $\frac{\partial}{\partial t}R\ge0$ by Harnack inequality. If the claim is not true, then  there exists  some constant $\varepsilon>0$ such that
\begin{align}
\frac{\partial}{\partial t}\Big(-\frac{1}{R(p,t)}\Big)\ge \varepsilon>0, ~\forall ~t\le0,
\end{align}
for all $p\in M$ that satifies $R(p,t)=M(t)$.   Let $F(t)=M(t)^{-1}$ and $G(p,t)=R(p,t)^{-1}$. For $\Delta t>0$, we have
\begin{align*}
\frac{d}{dt}F(t)=&\lim_{\Delta t\to0}\frac{F(t+\Delta t)-F(t)}{\Delta t}\\
\le&\lim_{\Delta t\to0}\frac{G(p,t+\Delta t)-G(p,t)}{\Delta t}\\
\le&-\varepsilon.
\end{align*}
It follows that
\begin{align}
\limsup_{t\to-\infty}(-t)M(t)\le \frac{1}{\varepsilon}.
\end{align}
It contradicts to (\ref{type I condition}). We complete the proof of the claim.

Let $(p_i,t_i)$ be a sequence of time-space points such that $R(p_i,t_i)=M(t_i)$ and $t_i\to-\infty$. By the claim, we also assume that $(p_i,t_i)$ satisfying the condition (\ref{quotient to zero}). Now, we consider rescaled  K\"{a}hler-Ricci flows $(M,g_i(t),p_i)$, where $g_i(t)=M(t_i)g(M(t_i)^{-1}t+t_i)$. Since $g_i(t)$ is uniformly bounded for $t\le0$ and is $\kappa$-noncollapsed, $(M,g_i(t),p_i)$ converges   to an ancient $\kappa$-solution $(M_{\infty},g_{\infty} (t), p_{\infty})$ in Cheeger-Gromov topology.   Moreover,  by  (\ref{quotient to zero}), we have
\begin{align}\label{limit zero}
\frac{\partial}{\partial t}R_{\infty}(p_{\infty},0)=&\lim_{i\to\infty}\frac{\partial [M(t_i)^{-1}R(p_i,M(t_i)^{-1}t+t_i)]}{\partial t}|_{t=0}\\
=&\lim_{i\to\infty}M(t_i)^{-2}\frac{\partial [R(p_i,M(t_i)^{-1}t+t_i)]}{\partial (M(t_i)^{-1}t)}|_{t=0}\notag\\
=&\lim_{i\to\infty}\frac{\partial}{\partial t}R(p_i,t_i)\cdot R^{-2}(p_i,t_i)\notag\\
=&0\notag,
\end{align}
and
\begin{align}\label{limit curvature}
R_{\infty}(p_{\infty},0)=1.
\end{align}
Note that by Theorem 2.1 in \cite{Ca2}, we may assume that $(M_{\infty},g_{\infty}(t),p_{\infty})$ has positive Ricci curvature. Thus,  by Corollary \ref{cao-theorem-1} together   with  (\ref{limit curvature}) and  (\ref{limit zero}),   the universal cover of $(M_{\infty},g_{\infty}(t))$ must be a $\kappa$-noncollpased and non-trivial steady K\"{a}hler-Ricci soliton. However, such a steady Ricci soliton doesn't exist by \cite{DZ3}.  Therefore,  (\ref{type I condition}) could not hold  and  we  prove the lemma.
\end{proof}

We need the following uniqueness result of   shrinking K\"{a}hler-Ricci solitons on $\mathbb{CP}^n$.

\begin{lem}\label{uniqueness}
If $(\mathbb{CP}^n,g)$ is a shrinking K\"{a}hler-Ricci soliton with nonnegative holomorphic bisectional curvature, then $g$ has constant holomorphic bisectional curvature.
\end{lem}
\begin{proof}
The lemma is already known by the uniqueness of K\"{a}hler-Ricci solitons on compact manifolds (cf. \cite{TZ1,TZ2}). In the following, we give another proof. We may assume  that $(\mathbb{CP}^n,g,f)$ is a shrinking K\"{a}hler-Ricci soliton for some smooth function $f\in C^{\infty}(\mathbb{CP}^n)$. Let $g(t)=\phi^{*}_tg$ and $\phi_t$ is generated by $-\nabla f$. We first note that the $(\mathbb{CP}^n,g(t))$ has positive Ricci curvature by Theorem 2.1 in \cite{Ca2}. If $(\mathbb{CP}^n,g)$ is symmetric, then it is K\"{a}hler-Eistein by the K\"{a}hler-Ricci soliton equation. Hence, $g$ has constant holomorphic bisectional curvature by the uniqueness of K\"{a}hler-Einstein metric (cf. \cite{BM}). If $(\mathbb{CP}^n,g)$ is not symmetric, then it must be irreducible and the holonomy group ${\rm Hol}(g(t))=U(n)$. By the argument in the proof of Theorem 1.2 in \cite{G}, we see that $(\mathbb{CP}^n,g(t))$ has positive holomorphic bisectional curvature.  By \cite{CT}, we see that $g(t)$ converge to a K\"{a}hler-Einstein metric with positive holomorphic bisectional curvature as $t\to\infty$. By  the convergence of $g(t)$ and the definition of $g(t)$,  we conclude that $g$ has constant holomorphic bisectional curvature. We complete the proof.
\end{proof}

By  Theorem \ref{Frankel conjecture}   and Lemma \ref{uniqueness},  we  prove

\begin{prop}\label{prop-type I case}
 Any backward Type I $\kappa$-solution of K\"{a}haler-Ricci flow on compact complex manifolds must be a quotient of products of irreducible compact Hermitian symmetric manifolds.
\end{prop}

\begin{proof}
Let $(M,g(t))$ be a backward Type I $\kappa$-solution of K\"{a}haler-Ricci flow on compact complex manifold $M$.  By Theorem \ref{Frankel conjecture}, the universal cover of $(M,g(0))$ is holomorphicly isometric to
\begin{align}
(\mathbb{C}^k,g_0)\times(\mathbb{CP}^{n_1},g_1)\times\cdots\times(\mathbb{CP}^{n_l},g_l)\times(M_1,h_1)\times\cdots\times(M_r,h_r)
\end{align}
We only need to show that $k=0$ and $g_i$ is holomorphicly isometric to the Fubini-Study metric on $\mathbb{CP}^{n_i}$ for $1\le i\le l$.

We first show that $k=0$. Fix $p\in M$, for any $\tau_i\to\infty$, let $g_{\tau_i}(t)=\tau_i^{-1}g(\tau_i t)$. Then, $(M,g_{\tau_i}(t),p)$ will converge to a shrinking Ricci soliton $(M_{\infty},g_{\infty}(t),p_{\infty})$ by taking a subsequence if necessary  (cf. \cite{Na}).  By the diameter estimate in \cite{Ni}, we see that $(M_{\infty},g_{\infty}(t))$ is compact for $t<0$ and $M_{\infty}$ is diffeomorphic to $M$. Hence, $M_{\infty}$ is covered by
\begin{align}
(\mathbb{C}^k,g_0)\times(\mathbb{CP}^{n_1},g_1)\times\cdots\times(\mathbb{CP}^{n_l},g_l)\times(M_1,h_1)\times\cdots\times(M_r,h_r)
\end{align}
Therefore, the fundamental group of $M_{\infty}$ is infinite if $k\ge 1$. However, the fundamental group of compact shrinking Ricci soliton $M_{\infty}$ is finite (cf. \cite{L}). Hence, we get $k=0$.

Now, we are left to show that $g_i$ is holomorphicly isometric to the Fubini-Study metric on $\mathbb{CP}^{n_i}$ for $1\le i\le l$. In this case, we may assume that $M=\mathbb{CP}^n$ for convenience. For fixed $p\in M$ and $\tau_i\to\infty$, we have $(M,g_{\tau_i}(t),p)$ will converge to a shrinking K\"{a}hler-Ricci soliton $(M_{\infty},g_{\infty}(t),p_{\infty})$ with nonnegative holomorphic bisectional curvature, where $g_{\tau_i}(t)=\tau_i^{-1}g(\tau_i t)$. We also note that $M_{\infty}$ is diffeomorphic to $M=\mathbb{CP}^n$. Therefore it is biholomorphic to $\mathbb{CP}^n$ by the generalized Frankel conjecture (cf. Theorem \ref{Frankel conjecture}). By Lemma \ref{uniqueness}, we know that $(M_{\infty},g_{\infty}(t))$ has positive constant holomorphic bisectional curvature for all $t<0$. By the convergence of $(M,g_{\tau_i}(t),p)$, we know that $g_{\tau_i}(t)$ also has positive holomorphic bisectional curvature for large $\tau_i$ and $t<0$. Since the positivity of bisectional cuvature is preserved along the flow, we know that $g(t)$ has positive holomorphic bisectional curvature for all $t$. By \cite{CT}, $g(t)$ blows up at some finite time $T$ and converge to a metric of  constant holomorphic bisectional curvature under rescaling as $t\to T$ (also see \cite{TZ3}). Note that the entropy invariant $\nu(M,g(t))$ is monotone along the flow.  By the forward and backward convergence of $g(t)$, the entropy invariant $\nu(M,g(t))$ is a constant along $g(t)$. Hence, $(M,g(t))$ is a shrinking K\"{a}hler-Ricci soliton with positive holomorphic bisectional curvature. Hence, $(M,g(t))$ has constant positive holomorphic bisectional curvature by Lemma \ref{uniqueness}.

\end{proof}

Theorem \ref{main theorem} follows  from Proposition \ref{prop-type I case} and Lemma \ref{lem-flow is type I} immediately.

\end{document}